\documentclass[12pt, letterpaper]{amsart}
\usepackage[utf8]{inputenc}
\usepackage{amssymb}
\usepackage{amsmath}
\usepackage{cite}
\usepackage{graphicx}
\usepackage{float}
\usepackage{hyperref}
\usepackage{cleveref}
\usepackage[margin=1in]{geometry}
\usepackage{mathtools}
\usepackage{mathrsfs}
\usepackage{xcolor}
\usepackage{lipsum}
\usepackage{mwe}
\usepackage{enumitem}

\newcommand{\CI}{\mathcal{C}^{\infty}}

\newcommand{\RR}{\mathbf{R}}

\newcommand{\restrictedto}{\big\rvert}

\newcommand{\loc}{\mathrm{loc}}
\newcommand{\comp}{\mathrm{comp}}

\newcommand{\RNum}[1]{\uppercase\expandafter{\romannumeral #1\relax}}

\newcommand{\df}{\coloneqq}
\newcommand{\la}{\langle}
\newcommand{\ra}{\rangle}

\newcommand{\RRS}{\RR^2\backslash S}
\newcommand{\tr}{\tilde{r}}

\newcommand{\WF}{\operatorname{WF}}
\newcommand{\aalpha}{\boldsymbol{\alpha}}
\newcommand{\pa}{P_{\alpha}}
\newcommand{\paa}{P_{\boldsymbol{\alpha}}}
\newcommand{\HH}{\mathcal{H}}

\newcommand{\RRB}{\RR^2\backslash B(0,R_0)}
\newcommand{\ft}{\mathcal{F}}
\newcommand{\usin}{\mathcal{U}}

\DeclareMathOperator{\im}{\text{Im}}
\DeclareMathOperator{\ima}{\emph{Im}}

\DeclareMathOperator{\rea}{\emph{Re}}

\DeclareMathOperator{\supp}{supp}

\DeclareMathOperator{\Id}{Id}

\newcommand{\dom}{\mathcal{D}}
\newcommand{\doma}{\mathcal{D}_{\boldsymbol{\alpha}}}

\newtheorem{theorem}{Theorem}[section]
\newtheorem*{theorem*}{Theorem}
\newtheorem*{corollary*}{Corollary}
\newtheorem{corollary}[theorem]{Corollary}
\newtheorem{lemma}[theorem]{Lemma}
\newtheorem{definition}[theorem]{Definition}
\newtheorem{proposition}[theorem]{Proposition}
\newtheorem{remark}[theorem]{Remark}

\title[Aharonov--Bohm Resolvent Estimates]{Resolvent estimates for the magnetic Hamiltonian with singular vector potentials and applications}
\author{Mengxuan Yang}
\date{\today}
\address{Department of Mathematics, Northwestern University, Evanston, IL, 60208, USA}
\email{mxyang@math.northwestern.edu}

\begin{document}

\maketitle
\begin{abstract}
    For the magnetic Hamiltonian with singular vector potentials, we analytically continue the resolvent to a logarithmic neighborhood of the positive real axis and prove resolvent estimates there. As applications, we obtain asymptotic locations of resonances and a local smoothing estimate for solutions of the corresponding Schr\"odinger equation.
\end{abstract}
\section{Introduction}
In this paper, we study the Aharonov--Bohm Hamiltonian with multiple solenoids
\begin{equation}
    \paa=\left(\frac{1}{i}\nabla-\Vec{A}\right)^2 \text{ on } X\df\RRS
\end{equation}
where $S=\{s_i\}_{i=1}^n\subset\RR^2$ and $\nabla\times \Vec{A}=\sum_{i=1}^{n}\alpha_i\delta_{s_i}$ with $\alpha_i\notin \mathbb{Z}$. We further make a hypothesis that \emph{no three solenoids are collinear}, which is generically true.

Our main theorem is the following: 
\begin{theorem}
\label{thm_main}
Fix $R_1>R_0$. For $\chi\in\CI_c(\RR^2)$ with $\supp \chi \subset B(0,R_1), \chi\rvert_{B(0,R_0)}\equiv 1$ where $S\subset B(0,R_0)$ and any $\epsilon>0$, there exists $M>0$ such that the cutoff resolvent 
$$\chi R_{\aalpha}(\lambda) \chi \df \chi(\paa-\lambda^2)^{-1}\chi$$
continues analytically from $\{\ima\lambda>0\}$ to the region:
\begin{equation}
    \left\{\lambda\in \mathbb{C}: \rea\lambda>M, \  \ima\lambda>-(\frac{1}{2d_{\max}}-\epsilon)\log|\lambda|\right\}
\end{equation}
where $d_{\max}$ is the maximum of distances between all pairs of solenoids. Moreover, $R_{\aalpha}(\lambda)$ satisfies the resolvent estimates
\begin{equation}
    \|\chi R_{\aalpha}(\lambda)\chi\|_{L^2\rightarrow \doma^j}\leq C_j|\lambda|^{j-1}e^{T|\ima\lambda|} \text{ for } j=0,1
\end{equation}
for some $C_j,T>0$ in the above region. 
\end{theorem}
In the above theorem, $\doma^s$ denotes the domain of $\paa^{s/2}$, where $\paa$ is taken as the Friedrichs extension. In particular, this theorem gives a logarithmic resonance-free region under the positive real axis. Combining it with the Corollary 7.2 in \cite{yang2021wave}, we immediately obtain the following corollary about asymptotic locations of resonances: 
\begin{corollary}
\label{corollary_res}
Let $\lambda_j$ be resonances (cf. Section \ref{section_res}) of Hamiltonian $\paa$ on $\RRS$. For any $\epsilon, \delta>0$, there exist $M>0, r(\delta)>0$ such that 
\begin{equation}
    \#\left\{\lambda_j: M<\rea\lambda_j<r, \  (\frac{1}{2d_{\max}}-\epsilon)\log|\lambda_j| \leq-\ima\lambda_j\leq (\frac{1}{2d_{\max}}+\epsilon)\log|\lambda_j|\right\}\geq r^{1-\delta}
\end{equation}
if $r>r(\delta)$, and
\begin{equation}
    \left\{\lambda_j: \rea\lambda_j>M, \ -\ima\lambda_j\geq (\frac{1}{2d_{\max}}-\epsilon)\log|\lambda_j|\right\}=\emptyset.
\end{equation}
\end{corollary}

\begin{figure}[H]
  \includegraphics[width=0.80\linewidth]{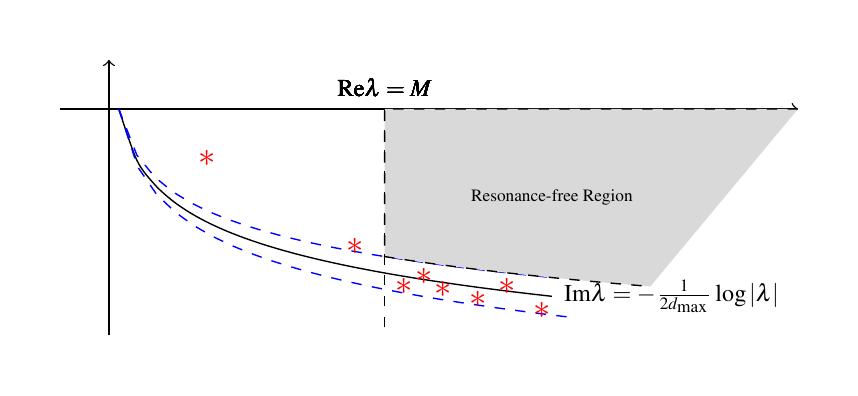}
  \caption{Resonances in a neighborhood of a logarithmic curve}
  \medskip
  \small
  \label{fig_res}
\end{figure}

\begin{remark}
The resonances in the above logarithmic neighborhood are generated by the weak trapping trajectories between the two farthest solenoids. In fact, the reason they are the resonances closest to the real axis as $|\lambda|\rightarrow\infty$ is because these weak trapping trajectories give the slowest rate of decay in the sense that they are diffracted the least frequently among all closed diffractive trajectories.

This corollary improves our previous knowledge \cite[Corollary 7.2]{yang2021wave} about resonances of the Aharonov--Bohm Hamiltonian with multiple solenoids.
\end{remark}

Another application of our resolvent estimate is a local-in-time local smoothing estimate without loss for Schr\"odinger equations of the Hamiltonian $\paa$, which may be of interest in the Strichartz estimates related to the Aharonov--Bohm Hamiltonian with multiple solenoids.
\begin{theorem}
Suppose $u$ solves the Schr\"odinger equation on $X$: 
\begin{equation}
    \begin{cases}
        (D_t-\paa)u=0 \\
         u\rvert_{t=0}=u_0\in L^2(X).
    \end{cases}
\end{equation}
Then for all $\chi\in \CI_c(X)$ with $\chi\rvert_{B(0,R_0)}\equiv 1$, $u$ satisfies the local smoothing estimates:
\begin{equation}
    \int_0^T\|\chi u\|_{\doma^{1/2}} dt\leq C_T \|u_0\|_{L^2}.
\end{equation}
\end{theorem}

The proof of Theorem \ref{thm_main} relies on the classic argument of Vainberg's parametrix \cite{vainberg1989asymptotic}. Tang--Zworski \cite{tang2000resonance} used a similar argument to show resolvent estimates and resonance-free region for the non-trapping black-box scattering. Baskin--Wunsch \cite{baskin2013resolvent} and Galkowski \cite{galkowski2017quantitative} later refined the argument and proved resolvent estimates with a logarithmic resonance-free region for manifolds with weak-trapping (a.k.a. diffractive trapping such as trapping generated by conic singularities) and Euclidean ends. In this paper, we generalize the resolvent estimates and the logarithmic resonance-free region to the Aharonov--Bohm Hamiltonian with multiple solenoids which has a non-trivial long range vector potential. Note that our argument also relies a weak non-trapping of singularities or a  \emph{very weak Huygens principle} (cf. \cite{baskin2013resolvent}), which is proved in Section \ref{section_WNT}. 

Another contribution of this work is a generalization of the black-box scattering with Euclidean end \cite{sjostrand1991complex} to the Hamiltonian with long range magnetic potentials. This generalization may be of further interest in investigating resonances of Hamiltonian with magnetic vector potentials generally. Within this framework, resonances in this paper are defined as poles of the meromorphic continuation of the resolvent $R_{\aalpha}(\lambda)$. Note that in the previous work \cite{yang2021wave}, resonances of $\paa$ are defined using the complex scaling. It is then also natural to discuss the equivalence of two definitions of resonances (cf. Section \ref{section_res}). To the best of our knowledge, this is the first time these arguments have been carried out for the magnetic Hamiltonian with vector potentials. 

The scattering resonances for the Aharonov--Bohm Hamiltonian with two largely separated solenoids were previously investigated by Alexandrova--Tamura \cite{alexandrova2011resonance} \cite{alexandrova2014resonances} using a new type of complex scaling method. They were also able to show approximate locations\footnote{A factor of 2 seems to be missing in the denominator of the expression of $\im\zeta_j(\lambda)$ there.} of high frequency resonances generated by the trapped trajectory between the two solenoids \cite[Corollary 1.2]{alexandrova2014resonances}. Our result coincides with theirs in the case of two solenoids. In fact, Corollary \ref{corollary_res} yields a more general result on asymptotic locations of resonances in the case of $n$ solenoids. 

The local smoothing effect without loss for solutions of Schr\"odinger equation holds under the non-trapping hypothesis (cf. \cite{ben1992decay} \cite{ben1998regularity} \cite{constantin1989local} \cite{ichi1996remarks} \cite{doi1996smoothing} \cite{kato1989some} \cite{burq2004nonlinear}). Doi \cite{doi1996smoothing} proved that on smooth manifolds the absence of trapped geodesic is necessary for the local smoothing estimates to hold without loss. Baskin--Wunsch \cite{baskin2013resolvent} showed that in the presence of diffractive trapping, such as closed diffractive geodesics in manifolds with conic singularities, the local smoothing estimate holds without loss. Duyckaerts \cite{duyckaerts2006inegalites} showed that local smoothing estimates hold for Schr\"odinger equation with inverse square potentials as a consequence of a resolvent estimate.
We show that for the diffractive trapping generated by singular vector potentials, the local smoothing estimate holds without loss. 
On the other hand, the study of local smoothing estimates for Schr\"odinger equations with the Aharonov--Bohm Hamiltonian with one solenoid has been carried out recently by Cacciafesta--Fanelli \cite{cacciafesta2019weak} and Gao \emph{et al.} \cite{GAO202170}, where they took full advantages of the scaling property in the radial direction. However, their approach no longer applies to our situation due to the failure of the scaling-invariance. To the best of our knowledge, this is also the first time that local smoothing estimates have been established for the Aharonov--Bohm Hamiltonian with multiple solenoids. The result may also be of interest in proving Strichartz estimates for Schr\"odinger equations. For local smoothing estimates and Strichartz estimates for solutions of wave, Dirac and Klein-Gordon equations with the Aharonov--Bohm Hamiltonian with single solenoid, see \cite{cacciafesta2017dispersive} \cite{cacciafesta2019weak} \cite{cacciafesta2020generalized} \cite{fanelli2020dispersive}.

The structure of the paper is the following: in Section \ref{section_prelim}, we introduce some preliminaries for the Aharonov--Bohm Hamiltonian and diffractive geometry; in Section \ref{section_res}, we generalize the black-box formalism to our situation and discuss the meromorphic continuation of the resolvent; in Section \ref{section_WNT}, we discuss the weak non-trapping of singularities generated by the diffraction; in Section \ref{Sec_res_est}, we prove Theorem \ref{thm_main} by constructing a Vainberg's parametrix; in Section \ref{section_lsm}, we prove the local smoothing estimate; in the Appendix \ref{app_ffs}, we discuss the forward fundamental solution of wave equations and the local wave decay which is used in Section \ref{Sec_res_est}. 

\subsection*{Future Directions}
Corollary \ref{corollary_res} gives resonances generated by the weakly trapped closed diffractive geodesic between two poles with the maximum distance. Similar to the situation in \cite{hillairet2020resonances}, these resonances correspond to resonant states concentrated at this orbit that loses energy to infinity via diffraction as infrequent as possible. On the other hand, we may expect each such trapped trajectory generates some logarithmic strings of resonances, although they may be diffracted more frequently than the ones in Corollary \ref{corollary_res}. In fact, we may expect that all high frequency resonances come from such classical weakly trapped closed diffractive geodesics, so there should be some resonance-free regions, or so-called spectral gaps (cf. \cite{dyatlov2016spectral}), between those strings of resonances. To show these further strings of resonances and spectral gaps, it requires finer microlocal analysis techniques to isolate resonant states with different decay rates. 

\subsection*{Acknowledgement}
The author is greatly indebted to Jared Wunsch for many instructive discussions as well as valuable comments on the manuscript. The author would also like to thank Luc Hillairet for proposing this interesting topic and providing helpful discussions. The author is also very grateful to Kiril Datchev, Plamen Stefanov, Junyong Zhang for many helpful discussions. The author would also like to thank an anonymous referee for various helpful suggestions on the manuscript. 

\section{Preliminaries}
\label{section_prelim}
In this section, we discuss preliminaries on the operator $\paa$ and the diffractive geometry in $\RR^2$ with $n$ solenoids. 

\subsection{Operators and domains}

We study the electromagnetic Hamiltonian 
\begin{equation}
\label{TotHam}
    \paa=\left(\frac{1}{i}\nabla-\Vec{A}\right)^2
\end{equation}
on the space $X\df\RRS$, where $S=\{s_i=(x_i,y_i)|1\leq i\leq n\}$ corresponds to locations of $n$ solenoids and $\Vec{A}=\sum_{i=1}^{n}\Vec{A}_i$ with 
\begin{equation}
\label{potential}
    \Vec{A_i}=-\alpha_i\cdot\left(-\frac{y-y_i}{(x-x_i)^2+(y-y_i)^2}, \frac{x-x_i}{(x-x_i)^2+(y-y_i)^2} \right)
\end{equation}
being the $i$-th vector potential corresponding to $s_i$, and $\aalpha=(\alpha_1,\cdots,\alpha_n)$ is the multi-index of the magnetic fluxes of all $n$ solenoids with $\alpha_i\notin\mathbb{Z}$. Note that the magnetic potential $\Vec{A}$ is singular at $s_i$ for $1\leq i\leq n$ and curl-free; therefore there is no magnetic field in $\RRS$. However, the motion of electrons in $\RRS$ can still ``feel" the influence of the magnetic field even though the electrons are completely shielded from the magnetic field. The electrons will experience a phase shift once we change the flux of the magnetic field, which can be observed by an interference experiment, although classically the change of magnetic flux has no influence on the motion of particles. This phenomenon generated by the singular magnetic potential $\Vec{A}$ is the so-called Aharonov--Bohm effect \cite{aharonov1959significance}, which suggests that the global electromagnetic vector potential is more physical than the local electromagnetic field in quantum mechanics.

Note that $P_{\aalpha}$ admits various self-adjoint extensions, for it is a positive symmetric operator defined on $\CI_c(X)\subset L^2(\RR^2)$ with deficiency indices $(2n,2n)$. In particular, we choose the Friedrichs self-adjoint extension, which corresponds to the following function space:
$$\text{Dom}(\paa^{\text{Fr}})\df \doma^2=\left\{u\in L^2: \paa u\in L^2, u\rvert_S=0\right\}.$$
Hereafter, $P_{\aalpha}$ denotes the Friedrichs extension of the Aharonov--Bohm Hamiltonian. For detailed discussions of self-adjoint extensions of the Aharonov--Bohm Hamiltonian, we refer to \cite{adami1998aharonov} and \cite{stovicek1998aharonov}. 

We define power domains for $s>0$ by 
$$\doma^s\df \left\{u \in L^2 : \paa^{s/2} u\in L^2\right\}$$
where $P^{s/2}_{\aalpha}$ is defined using the functional calculus. In particular, $s=1$ corresponds to the form domain and $s=2$ corresponds to the Friedrichs domain. Note that away from the singularity in $S$, the domain $\doma^s$ agrees with the usual Sobolev space $H^s$. 

\subsection{Diffractive geometry}
\label{DG}
Now we give a brief introduction of the diffractive geometry on $X=\RRS$. We only discuss what is relevant in proving the weak non-trapping of singularities in Section \ref{section_WNT}. For a more detailed discussion of diffractive geometry of $\paa$ on $\RRS$, we refer to \cite{yang2021wave}. 

Note that all \emph{regular geodesics} in $X$ are simply straight line segments in $X=\RRS$. Therefore, by standard propagation of singularities \cite{duistermaat1972fourier}, away from the solenoid set $S$, singularities of the wave equation propagate along the straight lines in $\RRS$. However, near each $s_i$ there are two types of \emph{generalized geodesics}, along which the singularities propagate, passing through $s_i$, which correspond to the diffractive and geometric waves emanating from the solenoid after the diffraction:

\begin{definition} 
Suppose $\gamma: [a,b]\rightarrow \RR^2$ is a continuous map whose image is a \emph{polygonal trajectory} with vertices at the solenoid set $S$.
\begin{itemize}
\item The polygonal trajectory $\gamma$ is a \emph{diffractive geodesic} if $\gamma^{-1}(S)$ is non-empty and $\gamma\big([a,b]\backslash\gamma^{-1}(S)\big)$ are finitely many straight line segments concatenated by the points in $S$.
\item The polygonal trajectory $\gamma$ is a (partially) \emph{geometric geodesic} if it is a diffractive geodesic and it contains a straight line segment parametrized by $[c,d]\subset[a,b]$ such that $\big(\gamma\restrictedto_{[c,d]}\big)^{-1}(S)$ non-empty, i.e., passing through at least one solenoid directly without being deflected. A geometric geodesic can be seen as the uniform limit of a family of regular geodesics in $X$.
\item In particular, the polygonal trajectory $\gamma$ is a \emph{strictly diffractive geodesic} if it is a diffractive geodesic but not a partially geometric geodesic.
\end{itemize}
\end{definition}

We now consider diffractive geodesics on the level of cosphere bundle $S^*X \cong (\RRS)\times S^1$ by taking them as (broken) \emph{integral curves} in $S^*X$ of the re-scaled Hamiltonian vector field $\mathbf{H}_{\aalpha}$ of the operator $\paa$. We may now define two symmetric relations between points in $S^*X$: a \emph{geometric} relation and a \emph{diffractive} relation. 

\begin{definition}
Let $q$ and $q'$ be points in the cosphere bundle $S^*X$.
\begin{enumerate}
    \item We define $q$ and $q'$ to be \emph{diffractively related by time $t$} if there exists a (broken) integral curve $\gamma\subset S^*X$ of $\mathbf{H_{\aalpha}}$ with length $t$, starting point $q$ and ending point $q'$. In particular, for each possible time $t'$ with jump in the $(\xi,\eta)$-variable, the ending point at $t'-$ and the starting point at $t'+$ must lie over the same point $s$ of $S$.    
    \item Among points that are diffractively related, we define $q$ and $q'$ to be \emph{geometrically related by time $t$} if there exists a continuous integral curve\footnote{With removable discontinuities over $S$.} $\gamma\subset S^*X$ of $\mathbf{H_{\aalpha}}$ with length $t$, starting point $q$ and ending point $q'$.
\end{enumerate}
\end{definition}

We can extend the diffractive/geometric relation to pairs of microlocal cutoffs $A_i, A_j\in \Psi_c^0(X)$, which are pseudodifferential operators of order zero with compact microsupport.

\begin{definition}
We say $A_i, A_j\in \Psi_c^0(X)$ are \emph{diffractively related} by time $t$ if there exists a (broken) integral curve $\gamma\subset S^*X$ such that $\gamma(0)\in \WF'(A_i)$ and $\gamma(t)\in \WF'(A_j)$, and \emph{geometrically related} by time $t$ if there exists a continuous integral curve $\gamma\subset S^*X$ such that $\gamma(0)\in \WF'(A_i)$ and $\gamma(t)\in \WF'(A_j)$.
\end{definition}

\section{Meromorphic continuation and complex scaling}
\label{section_res}
In this section, we show the resolvent $R_{\aalpha}(\lambda)\df (P_{\aalpha}-\lambda^2)^{-1}$ can be meromorphically continued to the logarithmic cover $\Lambda$ of the punctured plane $\mathbb{C}\backslash\{0\}$ as an operator from $\HH_{\comp}$ to $\dom_{\aalpha,\loc}$ (defined later). Furthermore, for the Aharonov--Bohm Hamiltonian $P_{\aalpha}$, we discuss the equivalence of the meromorphic continuation definition of resonances and the complex scaling definition of resonances employed in \cite{yang2021wave}. 

Although the materials developed in this section is more or less well-known to the community (cf. \cite[Appendix]{alexandrova2011resonance} for a meromorphic continuation argument and \cite{sjostrand1997trace} for the complex scaling), to the best of our knowledge, the systematic treatments regarding the Hamiltonians with Aharonov-Bohm vector potentials have been omitted in the literature. The materials developed in this section can be generalized to the scattering theory of more general long-range vector potentials where the magnetic field only exists in a compact region. The idea comes from a combination of Alexandrova--Tamura's treatment \cite[Appendix]{alexandrova2011resonance} on the meromorphic continuation of the resolvent for two solenoids and Sj\"ostrand--Zworski's work \cite{sjostrand1991complex} on the complex scaling of the black box scattering. The basic idea of the approach is to choose the \emph{correct} ``free" Hamiltonian to cancel off the long range effect generated by the vector potentials in $P_{\aalpha}$ so that an analogous argument to the black-box scattering is applicable.

Consider the operator
$$P_{\beta}=(-i\nabla-\Vec{A}_{\beta})^2$$ 
with the vector potential 
\begin{equation*}
    \Vec{A}_{\beta}=-\beta\cdot\left(-\frac{y-c_y}{(x-c_x)^2+(y-c_y)^2}, \frac{x-c_x}{(x-c_x)^2+(y-c_y)^2} \right)
\end{equation*}
with $\beta=\sum_{i=1}^n \alpha_i$ being the total flux and $c=(c_x,c_y)=\frac{1}{\beta} \sum_{i=1}^n \alpha_i\cdot (x_i,y_i)$ being the ``center of flux" of all solenoids if $\beta\neq 0$. In particular, when $\beta=0$, $P_{\beta}$ can be taken as the free Laplacian on $\RR^2$. Since the flux of $P_{\beta}$ is chosen to be the total fluxes of $P_{\aalpha}$, we can choose a real-valued function $f\in\CI(\RR^2)$ such that 
\begin{equation}
\label{P0}
    P_0\df e^{-if}P_{\beta}e^{if}
\end{equation}
equals $P_{\aalpha}$ when restricted to $\RR^2\backslash B(0,R_0)$, where $B(0,R_0)$ is the ball in $\RR^2$ contains all the solenoids.

We first recall some basic notions regarding the black-box scattering in our situation. For a detailed presentation in the scattering theory with Euclidean end, we refer to \cite[Section 4.1]{dyatlov2019mathematical}. Let $\HH_{\bullet}$, with $\bullet=\aalpha$ or $0$,
be complex Hilbert spaces with an orthogonal decomposition
\begin{equation}
    \HH_{\bullet} = \HH_{\bullet,R_0} \oplus L^2(\RR^2 \backslash B(0,R_0)),
\end{equation}
where $\HH_{\bullet,R_0}$ are Hilbert spaces. In fact, we have
\begin{equation*}
    \HH_{\bullet,R_0}=L^2(B(0,R_0)) \text{ and } \HH_{\bullet}=L^2(\RR^2).
\end{equation*}
We consider the Friedrichs self-adjoint extension of $P_{\bullet}$ with $\bullet=\aalpha$ or $0$, which denote the Hamiltonian with multiple solenoids and the ``free" Hamiltonian correspondingly. Then $P_{\bullet}$ has domain $\dom_{\bullet}\subset \HH_{\bullet}$ with
\begin{equation}
    1_{\RR^2 \backslash B(0,R_0)}\dom_{\bullet}= H^2(\RR^2 \backslash B(0,R_0)),
\end{equation}
\begin{equation}
    1_{\RR^2 \backslash B(0,R_0)}P_{\aalpha}= P_{0}\rvert_{\RR^2 \backslash B(0,R_0)},
\end{equation}
\begin{equation}
    1_{B(0,R_0)}(P_{\bullet}+i)^{-1} \text{ is compact }\HH_{\bullet}\rightarrow\HH_{\bullet}
\end{equation}
We further define 
\begin{equation*}
    \HH_{\bullet,\loc}\df \HH_{\bullet,R_0} \oplus H^2_{\loc}(\RR^2 \backslash B(0,R_0)),\ 
    \HH_{\bullet,\comp}\df \HH_{\bullet, R_0} \oplus H^2_{\comp}(\RR^2 \backslash B(0,R_0))
\end{equation*}
and 
\begin{equation*}
    \dom_{\bullet,\loc}\df \{u\in \HH_{\bullet,\loc}: \chi\in \CI_c, \chi\rvert_{B(0,R_0)}\equiv1 \Rightarrow \chi u\in \dom_{\bullet}\},\ 
    \dom_{\bullet,\comp}\df \dom_{\bullet} \cap \HH_{\bullet,\comp}.
\end{equation*}

Following \cite[Lemma 4.3]{dyatlov2019mathematical}, we have the following result: 
\begin{lemma}
\label{res_est}
For k=0,1,2 and $\tau>0$, the resolvent $R_{\bullet}(\lambda)\df (P_{\bullet}-\lambda^2)^{-1}$ satisfies the estimates
\begin{equation}
    \| 1_{\RRB}(P_{\bullet}-i\tau)^{-1}\|_{\HH_{\bullet}\rightarrow H^{k}(\RRB)}\leq C \la \tau \ra^{k/2} \tau^{-1}
\end{equation}
\end{lemma}
\begin{proof}
It follows directly from the self-adjointness of the Friedrichs extension and a Sobolev interpolation argument.
\end{proof}

The following proposition gives the analytic continuation of the free resolvent $R_0(\lambda)$. 
\begin{proposition}
\label{free_cont}
For the Hamiltonian $P_0$ given in \eqref{P0}, the resolvent $R_0(\lambda): \HH_0\rightarrow\HH_0$ for $\ima\lambda>0$ admits an analytic continuation to the logarithmic covering $\Lambda$ of $\mathbb{C}\backslash\{0\}$ as an operator $\HH_{0,\comp}\rightarrow \dom_{0,\loc}$.
\end{proposition}
\begin{proof}
Using separation of variables, we can write out the piecewise resolvent of $P_{\beta}$ explicitly using Bessel functions, then we prove the analytic continuation using the analytic continuation and asymptotics of Bessel functions.

Without loss of generality, we assume that the solenoid of $P_{\beta}$ is at the origin, i.e., $c=0$. using the Fourier series, we decompose $f\in L^2(\RR^2)$ into direct sums as follows:
$$L^2(\RR^2)=\bigoplus_{j=1}^{\infty} L^2(\RR_+; E_j);\ f(r,\theta)=\sum_{j=1}^{\infty} f_j(r)\cdot e^{ij\theta}.$$
For any $\lambda\in \mathbb{C}$ with $\arg\lambda\in (0,\pi)$, the resolvent $R_{\beta}(\lambda) \df (P_{\beta}-\lambda^2)^{-1}$ acting on $L^2(\RR^2)$ splits into the direct sum as follows: 
$$R_{\beta}(\lambda)f=\bigoplus_{j=1}^{\infty} (R_{\beta,j}(\lambda)f_j)\cdot e^{ij\theta},$$
where $R_{\beta,j}(\lambda)$ is the resolvent on each mode. We first compute the resolvent for $\im\lambda>0$. Using separation of variables and the direct sum decomposition above, the equation 
$$(P_{\beta}-\lambda^2)u=f$$
reduces to the following equation 
\begin{equation}
\label{bess}
 \partial_r^2 u_j+\frac{1}{r}\partial_ru_j-\frac{(j+\beta)^2}{r^2}u_j+\lambda^2u_j=f_j, \ j\in \mathbb{Z}.   
\end{equation}
Using change of variables to $\rho=\lambda r$ and writing $v(\rho)=u_j(\rho/\lambda)$, we obtain the Bessel equation:
$$v''+\frac{1}{\rho}v'+(1-\frac{\nu_j^2}{\rho^2})v=g(\rho)$$
where $\nu_j=|j+\beta|$ and $g(\rho)=\frac{1}{\lambda^2}f_j(\rho/\lambda)$. Using standard variation of parameters method in ODE, we may write the \emph{outgoing} solutions to \eqref{bess} as 
$$u_j(r)=\left(\int_r^{\infty}\frac{w_2(\tr)f_j(\tr)}{W(w_1,w_2)}d\tr\right)w_1(r)+\left(C+\int^r_0\frac{w_1(\tr)f_j(\tr)}{W(w_1,w_2)}d\tr\right)w_2(r)$$
where $w_1(r)=J_{\nu_j}(\lambda r)$ and $w_2(r)=H^{(1)}_{\nu_j}(\lambda r)$ are basis for the space of solutions to the homogeneous equation and $W(w_1,w_2)(r)=\frac{2i}{\pi} r^{-1}$ is the Wronskian. Then the Friedrichs extension near $r=0$ forces $C=0$, by the asymptotic expansion of Bessel functions near zero \cite[10.7]{NIST:DLMF} and the fact that $u_j$ and its derivative lie in the weighted $L^2$-space \cite[Section X.3]{reed1972methods}. This yields the resolvent kernel on the $j$-th mode
\begin{equation}
    R_{\beta,j}(\lambda)f_j (r) = \frac{\pi}{2i}\left(
    \int_r^{\infty} J_{\nu_j}(\lambda r) H^{(1)}_{\nu_j}(\lambda \tilde{r}) f(\tilde{r}) \tilde{r} d\tilde{r} +
    \int^r_0 H_{\nu_j}^{(1)}(\lambda r) J_{\nu_j}(\lambda \tilde{r}) f(\tilde{r}) \tilde{r} d\tilde{r}
    \right).
\end{equation}
In order to show $\chi R_{\beta}(\lambda)\chi$ analytically continues to $\Lambda$, it suffices to show that for any $f,g \in L^2(\RR^2)$, the function
\begin{equation}
    \la \chi R_{\beta}(\lambda)\chi f, g \ra
\end{equation}
analytically continues to $\Lambda$. See \cite[Section VI.3]{reed1972methods} for the operator valued analyticity. By analytical continuation of Bessel functions,
\begin{equation}
    \la \chi R_{\beta,j}(\lambda)\chi f_j, g_j \ra
\end{equation}
analytically continues to $\Lambda$. Therefore, it suffices to show the partial sum of the series
\begin{equation}
\label{res_decomp}
     \la \chi R_{\beta}(\lambda)\chi f, g \ra= \sum_j \la \chi R_{\beta,j}(\lambda)\chi f_j, g_j \ra
\end{equation}
uniformly converges in $j$. In fact, this can be shown by using the asymptotics of Bessel functions as $\nu\rightarrow\infty$ (cf. \cite[10.19]{NIST:DLMF}): 
\begin{equation}
    J_{\nu}\left(z\right)\sim\frac{1}{\sqrt{2\pi\nu}}\left(\frac{ez}{2\nu}\right)^{\nu},\     {H^{(1)}_{\nu}}\left(z\right)\sim-i\sqrt{\frac{2}{\pi\nu}}\left(\frac{ez}{2\nu}\right)^{-\nu},
\end{equation}
so that each term on the RHS of \eqref{res_decomp} is bounded by 
\begin{equation}
    C_{\chi} \nu_j^{-1} \|f_j\|_{L^2} \|g_j\|_{L^2}
\end{equation}
which is absolutely summable in $j$. 

This proves that the cutoff resolvent $\chi R_{\beta}(\lambda) \chi $ can be analytic continued to the logarithmic cover $\Lambda$. The gauge equivalence of $P_{\beta}$ and $P_0$ therefore gives the analytic continuation of $R_0(\lambda)=e^{-if}R_{\beta}(\lambda)e^{if}$ to $\Lambda$.
\end{proof}

One of the main results of this section is the following: 
\begin{theorem}
For the Hamiltonian $P_{\aalpha}$, the resolvent 
\begin{equation}
    R_{\aalpha}(\lambda): \HH_{\aalpha}\rightarrow\HH_{\aalpha}
\end{equation}
is meromorphic for $\ima\lambda>0$. Moreover, the resolvent extends meromorphically to the logarithmic cover $\Lambda$ as a family of operators 
\begin{equation}
    R_{\aalpha}(\lambda): \HH_{\aalpha,\comp}\rightarrow\dom_{\aalpha,\loc},\  \lambda\in\Lambda.
\end{equation}
\end{theorem}
\begin{proof}
We first consider $\im\lambda>0$. We choose $\chi_j\in \CI_c(\RR^2), j=0,1,2,3$ with 
$$\chi_j\equiv 1 \text{ for } x\in B(0,R_0+\epsilon),\ \chi_j=\chi_j\chi_{j+1} \text{ and } \supp \chi_j\subset B(0,R)$$
for some $R>R_0$. Then we choose $\lambda_0\in \mathbb{C}$ with $\im\lambda_0>0$ and define 
\begin{equation}
    Q(\lambda,\lambda_0)\df (1-\chi_0)R_0(\lambda)(1-\chi_1)+\chi_2R_{\aalpha}(\lambda_0)\chi_1
\end{equation}
Note that since $P_{\aalpha}(1-\chi_0)=P_{0}(1-\chi_0)$, we conclude that 
\begin{equation}
\label{res_id1}
    (P_{\aalpha}-\lambda^2)Q(\lambda,\lambda_0)=\Id+ K(\lambda,\lambda_0)
\end{equation}
with 
\begin{equation}
\label{K}
    K(\lambda,\lambda_0)=-[P_0,\chi_0]R_0(\lambda)(1-\chi_1)+[P_{\aalpha},\chi_2]R_{\aalpha}(\lambda_0)\chi_1+(\lambda_0^2-\lambda^2)\chi_2 R_{\aalpha}(\lambda_0)\chi_1.
\end{equation}
Note that for $\im\lambda>0$, $K$ is compact $\HH\rightarrow\HH$. By Lemma \ref{res_est} and a Neumann series argument, the inverse of $\Id+K(\lambda,\lambda_0)$ exists for $\lambda=\lambda_0=e^{i\pi/4}\mu$ with $\mu\gg1$. Therefore by the analytic Fredholm theory,
\begin{equation}
   (\Id+ K(\lambda,\lambda_0))^{-1}: \HH_{\aalpha}\rightarrow\HH_{\aalpha} \text{ for }\im\lambda>0
\end{equation}
is a meromorphic family of operators. Hence, equation \eqref{res_id1} shows that for $\im\lambda>0$,
\begin{equation}
\label{res_id2}
    R_{\aalpha}(\lambda)= Q(\lambda,\lambda_0)(\Id+ K(\lambda,\lambda_0))^{-1}
\end{equation} and the meromorphic continuation of the LHS to $\lambda\in\im\lambda>0$. 

Now we consider the meromorphic continuation of $R_{\aalpha}(\lambda)$ to $\Lambda$. Note that by Proposition \ref{free_cont}, $K(\lambda,\lambda_0)$ is compact $\HH_{\aalpha,\comp}\rightarrow\HH_{\aalpha,\comp}$ for $\lambda\in \Lambda$ with $\lambda_0$ chosen above. The same Neumann series argument shows that $\Id+ K(\lambda,\lambda_0)$ is invertible for $\lambda=\lambda_0$. Therefore, by the analytic Fredholm theory, $(\Id+ K(\lambda,\lambda_0))^{-1}$ is meromorphic in $\lambda\in\Lambda$ as a family of operators $\HH_{\aalpha,\comp}\rightarrow\HH_{\aalpha,\comp}$. Since 
$$Q: \HH_{\aalpha,\comp}\rightarrow \dom_{\aalpha,\loc}, \ \lambda\in \Lambda$$
is a meromorphic family of operators by Proposition \ref{free_cont}, by equation \eqref{res_id2}, we obtain the meromorphy of $R_{\aalpha}(\lambda)$ in $\Lambda$.
\end{proof}

We can thus define the \emph{resonances} of $P_{\aalpha}$ to be poles of the meromorphic continuation of the resolvent $R_{\aalpha}(\lambda)$ on the logarithmic cover $\Lambda$. Note that in a previous paper \cite{yang2021wave} of the author, the resonances are defined using the complex scaling method in a neighborhood of the positive real axis. The following theorems states that the two definition of resonances there are indeed equivalent. The proof follows word by word from \cite[Lemma 3.5; Proposition 3.6]{sjostrand1991complex} (see also \cite[Theorem 4.37; 4.38]{dyatlov2019mathematical}) if we consider the above black-box scattering set up, therefore it is omitted here. 

\begin{theorem}
Let $R_0<R_1$ be as defined in the black box Hamiltonian and $P_{\aalpha,\theta}$ is the complex-scaled Hamiltonian defined in \cite[Section 7]{yang2021wave}. If $\chi\in \CI_c(B(0,R_1))$ is equal to $1$ near $B(0,R_0)$, then
\begin{equation}
    \chi R_{\aalpha}(\lambda)\chi = \chi (P_{\aalpha,\theta}-\lambda^2)^{-1}\chi
\end{equation}
for $\lambda$ such that $\ima (e^{i\theta}\lambda)> 0$. Moreover, the spectrum of $P_{\aalpha,\theta}$ in $\mathbb{C}\backslash e^{-2i\theta[0,+\infty)}$ agrees with resonances satisfying $\ima e^{i\theta}\lambda>0$ up to multiplicities, i.e., 
$$m_R(\lambda)=m_{\theta}(\lambda),\ \ima (e^{i\theta}\lambda)>0$$
where $m_R(\lambda)$ is the multiplicity (cf. \cite[Section 4.5]{dyatlov2019mathematical}) of the resonance at $\lambda$ and $m_{\theta}(\lambda)$ is the multiplicity of the eigenvalue of $P_{\aalpha,\theta}$.
\end{theorem}

\section{Weak non-trapping of singularities}
\label{section_WNT}
In this section, using techniques already established in \cite{yang2021wave}, we prove a weak non-trapping of singularities by decomposing the wave propagator using a microlocal partition of unity as in  \cite{baskin2013resolvent} and \cite{yang2021wave}. 

Consider two propagators
\begin{equation}
    \usin(t)= \frac{\sin (t\sqrt{\paa})}{\sqrt{\paa}} \text{ and } 
    U(t)= e^{it\sqrt{\paa}}.
\end{equation}
If we can show
\begin{equation}
\label{VWHP1}
 \chi U(t) \chi: \doma^s\longrightarrow\doma^{s+\frac{N}{2}},\ \text{ for } |t|>t_N,   
\end{equation}
by the functional calculus we obtain that 
\begin{equation}
 \chi \usin(t) \chi: \doma^s\longrightarrow\doma^{s+\frac{N}{2}+1},\ \text{ for } t>t_N.   
\end{equation}
We only show \eqref{VWHP1} for $t>0$: the proof for $t<0$ follows from the same proof except using the backward propagation. Define $d_{\max}\df \max_{i\neq j}d(s_i,s_j)$.

\begin{theorem}
\label{thm_vwhp}
For any fixed $N\in \mathbb{N}$ and $t>t_N \df (N+1) d_{\max} + 4R_1+1$ and $s>0$, 
\begin{equation}
    \chi U(t) \chi: \doma^s\longrightarrow\doma^{s+\frac{N}{2}}. 
\end{equation}
where  $\chi\in\CI_c(\RR^2)$ with $\supp \chi \subset B(0,R_1)$ and $\chi\rvert_{B(0,R_0)}\equiv 1$ with $\{s_i\}\subset B(0,R_0)$.
\end{theorem}

\begin{proof}
To prove the theorem, we first decompose the cutoff propagator $\chi U(t) \chi$ into finitely many microlocal propagators, then we show this weak non-trapping of singularities for each microlocal propagator using the theory of FIOs and propagation of singularities.

We take a microlocal partition of unity of $\supp \chi$. We start by choosing $\psi_i\in \CI_c(X), 1\leq i\leq n$ to be smooth cutoff functions at each solenoid such that $\psi_i\equiv 1$ near a $\delta/2$-neighborhood of the solenoid $s_i$ and whose support is contained in a $\delta$-neighborhood of $s_i$. Multiplying by $\psi_i$ therefore localizes within $\delta$-neighborhood of the solenoid $s_i$. $\{\psi_i\}_{1\leq i\leq n}$ are called \emph{solenoid cutoffs}. Next, let $\{A_k\}_{1\leq k\leq m}\subset\Psi^0_c(X^{\circ})$ be a finite collection of pseudodifferential operators on the interior of $X$ satisfying the following properties: 
\begin{enumerate}
    \item The Schwartz kernel of each $A_k$ has compact support; 
    \item There exists a fixed constant $\delta_{\text{int}}$ such that  $\WF' A_k\subset S^*X$ is contained in a small ball of radius $\delta_{\text{int}}$ with respect to the induced metric on $S^*X$ (from $\RR^2$);
    \item The $A_k$'s complete the solenoid cutoffs to a microlocal partition of unity on $\supp \chi$ in the sense that 
    \begin{equation*}
        \WF'\left(\Id-\sum_{i=1}^n\psi_i-\sum_{k=1}^m A_k\right)\subset T^*(\RR^2\backslash\supp\chi).
    \end{equation*}
\end{enumerate}
These pseudodifferential operators $\{A_k\}$ are called \emph{interior microlocalizers}. Note that by adjusting the constant $\delta$ and $\delta_{\text{int}}$, we may choose the solenoid cutoffs/interior microlocalizers as small as we want while still being a finite family. We define $\{B_j\}_{1\leq j\leq n+m}\df \{A_k\}\cup \{\psi_i\}$ to be the complete set of microlocal partition of unity of $\supp \chi$. 

Now we show that if $\delta$ and $\delta_{\text{int}}$ are sufficient small, for $t>t_N$ and any $B_i, B_j$, the microlocalized propagator either has the mapping property
\begin{equation}
    B_jU(t)B_i: \doma^s \longrightarrow \doma^{s+\frac{N}{2}},
\end{equation}
or it enjoys a better mapping property in the sense that
\begin{equation}
    B_jU(t)B_i: \doma^{-\infty}\longrightarrow\doma^{\infty},
\end{equation}
i.e., it is a smoothing operator. 

As proved in \cite[Lemma 4.4]{yang2021wave}, the microlocalized propagator $B_jU(t)B_i: \doma^{-\infty}\rightarrow\doma^{\infty}$ if there are no diffractive geodesic $\gamma$ with $\gamma(0)\in \WF'(B_i)$ and $\gamma(t)\in \WF'(B_j)$. Therefore, we only consider the microlocalized propagators $B_jU(t)B_i$ for which $B_i$ and $B_j$ are diffractively related (cf. Section \ref{section_prelim}) for some time $T>t_N$. Note that there are four scenarios, since $B_i, B_j$ can be taken as either microlocal cutoffs $A_{\bullet}\in \Psi^0_c(X)$ or solenoid cutoffs $\psi_{\bullet}\in \CI_c(X)$. We only consider the cases of $A_jU(t)A_i$ and $A_jU(t)\psi_i$, the other two cases are similar. 

\emph{Case 1.} Consider microlocalized propagator $A_jU(T)A_i$ where $A_i,A_j\in \Psi^0_c(X)$ with their microsupport small enough (diameter less than $2\delta_{\text{int}}$). Since $T>t_N=(N+1) d_{\max} + 4R_1+1$ and $\WF'(A_k)\subset T^*(\supp\chi)\subset T^*B(0,R_1)$ for all $1\leq k\leq m$, for any diffractive geodesic with length $T$ starting in $\WF'(A_i)$ and ending in $\WF'(A_j)$, there exists at least $N$ diffractions along it. Moreover, by the finiteness of diffractive geodesics with fixed endpoints and length, there are at most finitely many diffractive geodesics starting in $\WF'(A_i)$ and ending in $\WF'(A_j)$ (modulo variation of endpoints in $\WF'(A_i)$ and $\WF'(A_j)$). Therefore, we may insert microlocal cutoffs $A^{(k)}_{\bullet}\in \Psi^0_c(X)$ and obtains
\begin{equation}
    A_jU(T)A_i\equiv \sum_{k} A_jU(T-s_{k,N_k-1}-\cdots - s_{k,1})A^{(k)}_{N_k-1}U(s_{k,N_k-1})\cdots U(s_{k,2})A^{(k)}_1U(s_{k,1})A_i
\end{equation}
modulo smoothing errors, where each propagator $U(s_{\bullet})$ or $U(T-\sum s_{\bullet})$ involves exactly one diffraction, $N_k\geq N$ and each $A^{(k)}_{\bullet}\in \Psi^0_c(X)$ is a microlocal cutoff in the partitions of unity $\{B_j\}$. (Cf. \cite[Lemma 4.4]{yang2021wave}.) Applying Proposition 4.2 from \cite{yang2021wave} and mapping properties of FIOs (cf. \cite[Section 25.3]{MR2512677}) to each term in the above decomposition yields that
\begin{equation}
    A_jU(t)A_i: \doma^s \longrightarrow \doma^{s+\frac{N}{2}}.
\end{equation}

\emph{Case 2.} Now we treat the terms like $A_jU(T)\psi_i$ with a solenoid cutoff on one side. Since $\supp \psi_i\subset B(0,\delta)$ for $\delta \ll 1$, we may choose $\epsilon > 4\delta$ such that
\begin{equation}
    A_jU(T)\psi_i\equiv A_jU(T-\epsilon)\left(\sum_{1\leq l\leq m} A_l\right) U(\epsilon)\psi_i
\end{equation}
modulo smoothing terms, where the $\epsilon$-time propagation propagates $\WF'(\psi_i)$ to the outside of $B(s_i,\delta)$ and the smallness of $\epsilon$ guarantees the $\epsilon$-time propagation from $s_i$ not reaching other solenoids. Then the mapping property
\begin{equation}
    A_jU(t)\psi_i: \doma^s \longrightarrow \doma^{s+\frac{N}{2}}
\end{equation}
follows from the unitarity of the propagator $U(\epsilon)$ and the first case.
\end{proof}

\section{Resolvent estimates and resonance-free region}
\label{Sec_res_est}
In this section, adapting Vainberg's parametrix \cite{vainberg1989asymptotic} to our setting, we prove the resolvent estimate and give a resonance-free region assuming the very weak Huygens' principle (VWHP, a.k.a. weak non-trapping of singularities) proved in the last section: for any $s\geq0$ and $t>t_{N}$,
$$\chi\mathcal{U}(t)\chi: \doma^s\longrightarrow \doma^{s+\frac{N}{2}+1}.$$

We first record a lemma which will be used in the proof of Theorem \ref{thm1}. For a detailed proof, we refer the reader to \cite{baskin2013resolvent}.

\begin{lemma}
\label{lemma_ft_est}
Suppose $H_1$ and $H_2$ are Hilbert spaces and $N(t): H_1\rightarrow H_2$ is a family of bounded operators that have $k$ continuous derivatives in $t$ when $t\in \RR$, depending analytically on $t$ when $\rea t>T>0$, and equal to zero when $t<0$. Suppose that there are constants $j_0,\ k\geq j_0+2$, and $C_j$ such that for all $0\leq j\leq k$,
\begin{equation}
    \|\frac{\partial^j}{\partial t^j}N(t)\|\leq C_j |t|^{j_0-j}\ \text{ for } \rea t>T.
\end{equation}
Then the operator 
\begin{equation*}
    \check{N}(\lambda)=\ft^{-1}_{t\rightarrow\lambda}(N(t)): H_1\longrightarrow H_2\ \text{ for }\ima\lambda>0
\end{equation*}
can be analytically continued to the domain $-\frac{3\pi}{2}\leq \arg\lambda \leq \frac{\pi}{2}$, and when $|\lambda|>1$, it satisfies the estimate
\begin{equation}
    \|\check{N}(\lambda)\|\leq C_j|\lambda|^{-j}e^{T|\ima\lambda|}\ \text{ for } j=0, \dots, k.
\end{equation}
\end{lemma}

Now we state and prove the main theorem of this section:

\begin{theorem}
\label{thm1}
Fix $R_1>R_0$. Suppose that for all $N>0$, there exist $t_N>0$ such that for all $\chi\in\CI_c(\RR^2)$ with $\supp \chi \subset B(0,R_1)$ and $s\geq 0$
\begin{equation}
\label{VWHP_N}
    \chi\mathcal{U}(t)\chi: \doma^s\longrightarrow\doma^{s+N+1}, \ t>t_N.
\end{equation}
Define
\begin{equation*}
    T_N\df \inf \{t_N>0: \eqref{VWHP_N} \text{ holds for } t_N\}
\end{equation*}
and 
\begin{equation*}
    \bar{T}\df \limsup_{N\rightarrow\infty}\frac{T_N}{N}>0.
\end{equation*}
Then for all $\epsilon>0$, there exists $M>0$ such that for all $\chi\in\CI_c(\RR^2)$ with $\supp\chi \subset B(0,R_1)$, the cutoff resolvent 
$$\chi R_{\aalpha}(\lambda) \chi \df \chi(\paa-\lambda^2)^{-1}\chi$$
continues analytically to the region:
\begin{equation}
\label{log_region}
    \{\lambda\in \mathbb{C}: \rea\lambda>M, \  \ima\lambda>-(\bar{T}^{-1}-\epsilon)\log|\lambda|\}
\end{equation}
Moreover, $R_{\aalpha}(\lambda)$ satisfies the resolvent estimates
\begin{equation}
\label{Resolvent_Est}
    \|\chi R_{\aalpha}(\lambda)\chi\|_{\HH_{\aalpha}\rightarrow \doma^j}\leq C_j|\lambda|^{j-1}e^{T|\ima\lambda|} \text{ for } j=0,1
\end{equation}
for some $C_j,T>0$ in the above region.
\end{theorem}

\begin{proof}
First we construct a parametrix of the resolvent $R_{\aalpha}(\lambda)$: 
\begin{equation*}
    R^{\sharp}(\lambda)= R_1(\lambda) + R_2(\lambda)
\end{equation*}
where $R_1(\lambda)$ approximates the resolvent in the interaction region and $R_2(\lambda)$ approximates the resolvent in the free region. Consider the identity
\begin{equation*}
    R_{\aalpha}(\lambda)=\int_0^{\infty}\usin(t)e^{it\lambda}dt=\ft^{-1}(H(t)\usin(t)).
\end{equation*}
where $H(t)$ is the Heaviside function. Fix $N>0$ arbitrary, then by assumption, there exists $T_N>0$ such that for all $s\geq 0$
\begin{equation}
    \chi\usin(t)\chi: \doma^s\longrightarrow\doma^{s+N+1}, \ t>T_N, 
\end{equation}
where $\chi\in \CI_c(\RR^2)$ with $\chi\rvert_{B(0,R_0)}\equiv 1$ and $\supp \chi \subset B(0,R_1)$. Define a family of cutoff fuctions $\chi_i\in \CI_c(B(0,R_1)), i=0,1,2,3$ such that $\chi_1=\chi$ and $\chi_i\chi_{i+1}=\chi_{i+1}$. 

To define the approximated resolvent $R_1(\lambda)$, we introduce a spacetime cutoff function $\rho(t,z)\in \CI(\RR_t\times \RR^2_z)$ such that $0\leq \rho\leq 1,\ \rho\rvert_{B(0,R_0)}=\rho(t)$ and 
\begin{equation}
    \rho(t,z)=
    \begin{cases}
        1\ &t\leq |z|+T_N\\
        0\ &t\geq |z|+T_N+R_0+1
    \end{cases}
\end{equation}

\begin{figure}[H]
  \includegraphics[width=0.6\linewidth]{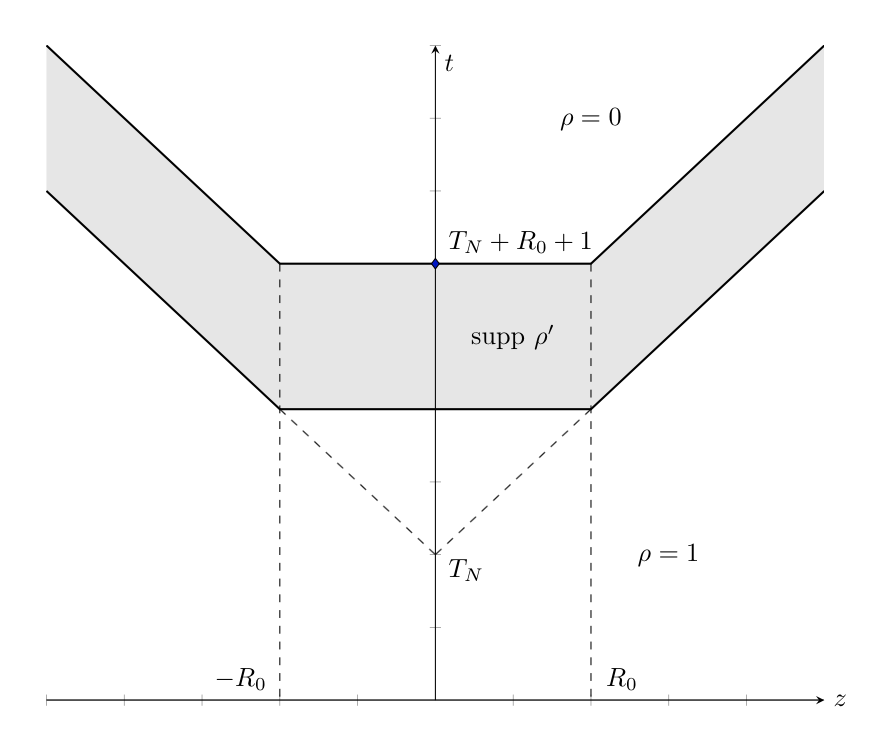}
  \caption{Support of functions $\rho$ and $\rho'$.}
  \medskip
  \small
  \label{fig_rho}
\end{figure}

Then by propagation of singularities and the VWHP, we have 
\begin{equation*}
    (1-\rho)\usin(t)\chi: \doma^s\longrightarrow\doma^{s+N+1},\ t\geq 0.
\end{equation*}
For $g\in L^2=\doma^0$, consider $\rho\usin(t)\chi g$. Define the wave operators $\Box\df D_t^2-\paa$ and $\Box_0\df D_t^2-P_0$. We can see that $\rho\usin(t)\chi g$ solves the wave equation 
\begin{equation}
    \begin{cases}
        \Box(\rho\usin(t)\chi g) = [\Box,\rho]\usin(t)\chi g\\
        \usin(0)\chi g=0,\ D_t\usin(0)\chi g=\chi g
    \end{cases}
\end{equation}
We define $F(t)g\df [\Box,\rho]\usin(t)\chi g = -(1-\rho)\usin(t)\chi g$. Then the VWHP and the support property of $\rho$ implies that 
\begin{equation}
\label{Fg}
    F(t)g\in \mathcal{C}^0(\RR_t, \doma^{N-1})\cap \mathcal{C}^{N-1}(\RR_t, L^2).
\end{equation}
We can now define the approximation near the interaction region by: 
\begin{equation}
    R_1(\lambda)\df \ft^{-1}_{t\rightarrow\lambda}(\rho H(t)\usin(t)\chi).
\end{equation}
Note that this is well-defined since for each $z\in \RR^2$, $\rho$ has compact support. Moreover, consider $\supp [\Box, \rho]$, then we have 
\begin{equation*}
    \supp F(t)g \subset\{|z|+T_N\leq t\leq |z|+T_N+R_0+1\}.
\end{equation*}
Due to the equation
\begin{equation}
    D_t^2(H(t)\usin(t)\chi g)=-i\delta(t)\chi g + H(t)D_t^2\usin(t)\chi g,
\end{equation}
we have
\begin{equation}
\begin{split}
    (\paa-\lambda^2)R_1(\lambda)g 
    &=-\ft^{-1}(\Box(\rho H(t)\usin(t)\chi g))\\
    &=-\ft^{-1}(F(t)g)+ i\chi g.
\end{split}
\end{equation}
Now we construct $R_2(\lambda)$ to further approximate the resolvent in the free region. We split $F(t)$ as 
\begin{equation*}
\begin{split}
    F(t)g 
    &= \chi_2F(t)g+ (1-\chi_2)F(t)g\\
    &=: F_1(t)g + F_2(t)g.
\end{split}
\end{equation*}
Note that by construction $P_0\equiv \paa$ as an operator on  $L^2(\RR^2\backslash B(0,R_0))$. Now we find $V(t)g$ such that it solves the following inhomogeneous equation
\begin{equation}
    \begin{cases}
        \Box_0(V(t) g) = -F_2(t)g\\
        V(0) g=0,\ D_t V(0) g=0
    \end{cases}
\end{equation}
Using the fact $\chi_3(1-\chi_2)\equiv 0$, we observe that 
\begin{equation*}
    \begin{split}
        -F_2(t)g 
        &= \Box_0 (V(t)g)\\
        &= \Box_0(\chi_3V(t)g)+\Box_0((1-\chi_3)V(t)g)\\
        &= [\Box_0,\chi_3]V(t)g + \Box_0(1-\chi_3) V(t)g. 
    \end{split}
\end{equation*}
We define the approximation of $R_{\aalpha}(\lambda)$ near the free end by
\begin{equation}
    R_2(\lambda)\df -\ft^{-1}_{t\rightarrow\lambda}((1-\chi_3)V(t)).
\end{equation}
In particular, noting that $\Box\equiv\Box_0$ on $\supp(1-\chi_3)$, 
\begin{equation}
\begin{split}
    (\paa-\lambda^2)R_2(\lambda)g
    &=\ft^{-1}(\Box (1-\chi_3)V(t)g)\\
    &=\ft^{-1}([P_0,\chi_3]V(t)g+(1-\chi_2)F(t)g).
\end{split}
\end{equation}
Defining the approximation
$R^{\sharp}(\lambda)\df R_1(\lambda)+ R_2(\lambda)$, we have
\begin{equation}
\label{res_id3}
    \begin{split}
        (\paa-\lambda^2) R^{\sharp}(\lambda)g
        &= \ft^{-1}(-F(t)g)+ i\chi g + \ft^{-1}([P_0,\chi_3]V(t)g+(1-\chi_2)F(t)g)\\
        &= i\chi \big(\Id +i\ft^{-1}(\chi_2F(t)-[P_0,\chi_3]V(t))\big)g
    \end{split}
\end{equation}
At the formal level, the cutoff resolvent can be written as 
\begin{equation}
    i\chi R_{\aalpha}(\lambda) \chi= \chi R^{\sharp}(\lambda) \big(\Id +i\ft^{-1}(\chi_2F(t)-[P_0,\chi_3]V(t))\big)^{-1} =: \chi R^{\sharp}(\lambda) (\Id+ K(\lambda))^{-1}
\end{equation}
Therefore, to analytically continue the resolvent, we only need to determine the region in which $\Id+K(\lambda)$ is invertible. We need the following lemma:
\begin{lemma}
\label{lemma_estimates}
Let $T'_N= T_N+R_0+3R_1$, then
\begin{equation}
\label{est1}
    \|\ft^{-1}(\chi_2 F(t))\|_{L^2\rightarrow L^2}\leq C_{j,1} |\lambda|^{-j}e^{T'_N|\ima \lambda|} \ \text{ for } j=0, 1, \dots, N-1
\end{equation}
\begin{equation}
\label{est2}
    \|\ft^{-1}([P_0,\chi_3]V(t))\|_{L^2\rightarrow L^2}\leq C_{j,2} |\lambda|^{-j}e^{T'_N|\ima \lambda|} \ \text{ for } j=0, 1, \dots, N-1
\end{equation}
\begin{equation}
\label{est3}
    \|\chi R_1(\lambda)\|_{L^2\rightarrow \doma^j}\leq C_{j,3} |\lambda|^{j-1}e^{T'_N|\ima \lambda|} \ \text{ for } j=0, 1
\end{equation}
\begin{equation}
\label{est4}
    \|\chi R_2(\lambda)\|_{L^2\rightarrow \doma^j}\leq C_{j,4} |\lambda|^{j-1}e^{T'_N|\ima \lambda|} \ \text{ for } j=0, 1
\end{equation}
\end{lemma}
We postpone the proof of the lemma. Assuming the lemma, we show the analytic continuation and the resolvent estimate for $R_{\aalpha}(\lambda)$. The estimates \eqref{est3} and \eqref{est4} imply 
\begin{equation}
    \|\chi R^{\sharp}(\lambda)\chi\|_{L^2\rightarrow \doma}\leq C_{j} |\lambda|^{j-1}e^{T'_N|\im \lambda|} \ \text{ for } j=0, 1.
\end{equation}
To show the resolvent estimate \eqref{Resolvent_Est} for $R_{\aalpha}(\lambda)$, by a Neumann series argument, we only need to show
\begin{equation}
    \|\ft^{-1}(\chi_2F(t)+[P_0,\chi_3]V(t))\|_{L^2\rightarrow L^2}\leq 1-\delta
\end{equation}
for some $\delta>0$ small. By estimates \eqref{est1} and \eqref{est2}, We only need to find $\lambda$ such that $|\lambda|>1$ and 
\begin{equation*}
    C'|\lambda|^{1-N}e^{T'_N|\im\lambda|}<1-\delta.
\end{equation*}
This reduces to 
\begin{equation}
\label{log_region1}
    |\im\lambda|<\frac{(N-1)\log |\lambda|+\log C_{\delta}}{T'_N}.
\end{equation}
The resolvent then satisfies the estimates
\begin{equation}
    \|\chi R_{\aalpha}(\lambda)\chi\|_{L^2\rightarrow \doma}\leq C_{j}\delta^{-1} |\lambda|^{j-1}e^{T'_N|\im \lambda|} \ \text{ for } j=0, 1
\end{equation}
in the region \eqref{log_region1}. Taking $N\rightarrow\infty$, we obtain the region of analytic continuation \eqref{log_region} from \eqref{log_region1}, where the constant $M$ there depends on $C_{\delta}$ in \eqref{log_region1} and $\epsilon$ in \eqref{log_region}. This finishes the proof the theorem.
\end{proof}

\begin{proof}[Proof of Lemma \ref{lemma_estimates}]
First we show \eqref{est1} and \eqref{est3}. They can be proved using energy estimates and integration by parts. 

For \eqref{est3}, recall that 
\begin{equation*}
    \chi R_1(\lambda)=\int_0^{\infty}\chi \rho \usin(t)\chi e^{it\lambda}dt.
\end{equation*}
Note that for $t>T'_N$, $\rho\chi\equiv0$. The the energy estimates $\|\usin(t)\chi\|_{L^2\rightarrow\doma^1}\leq C$ yields the estimates \eqref{est3} for $j=1$. Using integration by parts in $t$ yields the desired estimates for $j=0$. 

For \eqref{est1}, recall that 
\begin{equation*}
    \supp \chi F(t)\subset \{T_N\leq t\leq T'_N\}
\end{equation*}
and the VWHP implies \eqref{Fg}. Integration by parts $N-1$ times in $t$ yields the estimate \eqref{est1}. 

Now we prove \eqref{est2} and \eqref{est4}, using the forward fundamental solution of the ``free" wave equation and the Duhamel's principle. 

We make the following decomposition of $V(t)$: 
\begin{equation}
\label{V}
    V(t)g=-(1-\chi_2)\rho H(t)\usin(t)\chi g + H(t) \usin_0(t) (1-\chi_2)\chi g + q(t,z)
\end{equation}
where $\usin_0(t)=\frac{\sin{t\sqrt{P_0}}}{\sqrt{P_0}}$ is the ``free" propagator. Then we have
\begin{equation}
\label{q}
    \begin{split}
        \Box_0q
        &= -F_2(t)g +\Box_0\big((1-\chi_2)\rho H(t)\usin(t)\chi g\big) - \Box_0\big(H(t) \usin_0(t) (1-\chi_2)\chi g\big)\\
        &= -F_2(t)g + [P_0,\chi_2]\rho H(t)\usin(t)\chi g- (1-\chi_2)i\chi g + (1-\chi_2)\Box(\rho\usin(t)\chi g) + (1-\chi_2)i\chi g\\
        &= [P_0,\chi_2]\rho H(t)\usin(t)\chi g
    \end{split}
\end{equation}
with $q(0,z)=D_tq(0,z)=0$ and 
\begin{equation*}
    \supp [P_0,\chi_2]\rho H(t)\usin(t)\chi g \subset \{0\leq t\leq T_N+R_0\}.
\end{equation*}
We use the decomposition \eqref{V} to show \eqref{est2} and \eqref{est4} by verifying the estimate term-by-term. By equation \eqref{q} and Duhamel's principle
\begin{equation}
    \chi(z)q(t,z)=\int_0^t\chi(z) \usin_0(t-s)([P_0,\chi]\rho\usin(s)\chi g) ds
\end{equation}
By propagation of singularities of the wave equation for $P_0$ (cf. \cite{yang2021ab}) for $t>T'_N$, the inhomogeneous term is disjoint from the singular support of the Schwartz kernel of $\usin_0(t)$. Therefore, it is $\CI$ for $t>T'_N$ and satisfies decaying estimates
\begin{equation}
\label{est5}
    \|D_t^j(\chi q)\|_{\doma^1}\leq C_j t^{j_0-j}\|g\|_{L^2} \ \text{ for } j\in \mathbb{N}
\end{equation}
by the forward fundamental solution and its local decay computed in the Appendix \ref{app_ffs}. 

Similarly, the second term on the RHS of \eqref{V} is written as
\begin{equation}
    \chi \usin_0(t) (1-\chi_2)\chi g
\end{equation}
which also $\CI$ with polynomial bounds on its derivatives for $t>T'_N$. In particular, 
\begin{equation}
\label{est6}
    \|D_t^j(\chi \usin_0(t) (1-\chi_2)\chi g)\|_{\doma^1}\leq C_j t^{j_0-j}\|g\|_{L^2} \ \text{ for } j\in \mathbb{N}.
\end{equation}

To show \eqref{est2}, note that $[P_0,\chi_3](1-\chi_2)\equiv 0$ so $[P_0,\chi_3]$ vanishes the first term on the RHS of \eqref{V}. The VWHP implies that $V(t)\in \mathcal{C}^{N-1}(\RR_t,L^2)$. Applying Lemma \ref{lemma_ft_est} to \eqref{est5} and \eqref{est6} yields the desired estimate \eqref{est2}.

For the estimate \eqref{est4}, again we use the decomposition \eqref{V}. The estimate of $\ft^{-1}\big(\chi(1-\chi_2)\rho H(t)\usin(t)\chi g\big)$, which corresponds to the first term on the RHS of \eqref{V}, follows from \eqref{est3} while the estimates of the remainder terms are exactly same as before. This finishes the proof of the lemma. 
\end{proof}

\begin{proof}[Proof of Theorem \ref{thm_main}]
We only need to compute the constant $\bar{T}$ in Theorem \ref{thm1}. Theorem \ref{thm_vwhp} yields that
\begin{equation}
    \bar{T}=\lim_{N\rightarrow\infty} \frac{(N+1)d_{\max}+4R_1+1}{N/2}= 2d_{\max}.
\end{equation}
\end{proof}

\section{A Local smoothing estimate}
\label{section_lsm}
In this section, using the resolvent estimate \eqref{Resolvent_Est} established in Section \ref{Sec_res_est}, we prove a local smoothing estimate. It is expect to have further applications in proving Strichartz estimates for the Aharonov--Bohm Hamiltonian with multiple solenoids, although we will not discuss further in this paper. 

\begin{theorem}
Suppose $u$ satisfies the Schr\"odinger equation on $X$: 
\begin{equation}
    \begin{cases}
        (D_t-\paa)u=0 \\
         u\rvert_{t=0}=u_0\in L^2(X).
    \end{cases}
\end{equation}
Then for all $\chi\in \CI_c(X)$ with $\chi\rvert_{B(0,R_0)}\equiv 1$, $u$ satisfies the local smoothing estimates:
\begin{equation}
\label{local_smoothing}
    \int_0^T\|\chi u\|_{\doma^{1/2}} dt\leq C_T \|u_0\|_{L^2}.
\end{equation}
\end{theorem}

\begin{proof}
This is a direct corollary of our resolvent estimates Theorem \ref{thm1}, using the Plancherel theorem and the standard $TT^*$-argument (cf. \cite{burq2004nonlinear} \cite{burq2004smoothing}). We include the proof here for the sake of completeness. 

By Theorem \ref{thm1}, take $R(\lambda)=(\paa - \lambda)^{-1}$, then we know that 
\begin{equation}
\label{HFRS}
    \|\chi R(\lambda) \chi\|_{\doma^{-1/2}\rightarrow \doma^{1/2}}\leq C \text{ for } \lambda > M.
\end{equation}
Consider the inhomogeneous equation $(D_t-\paa)u=\chi f$, where $\chi f$ is compactly supported in $t$ with $\chi f\in L^2(\RR_t;\doma^{-1/2})$. Taking the Fourier transform in $t$ yields that
\begin{equation}
    (\lambda-\paa)\hat{u}=\chi\hat{f},
\end{equation}
so the high frequency resolvent estimates \eqref{HFRS} implies that 
\begin{equation}
    \|\Pi_{\infty}(\lambda)\chi \hat{u}\|_{L^2_{\lambda}\doma^{1/2}}\leq C \|\chi \hat{f}\|_{L^2_{\lambda}\doma^{-1/2}}
\end{equation}
where $\Pi_{\infty}$ is a smooth spectral cutoff near high frequency $\lambda >M$. Using Plancherel and the functional calculus, we obtain the imhomogeneous estimate
\begin{equation}
\label{inhomo_est}
    \|\Pi_{\infty}\chi u\|_{L^2_t\doma^{1/2}}\leq C \|\chi f\|_{L^2_t\doma^{-1/2}}.
\end{equation}
Now we can use a standard $TT^*$-argument to get the local smoothing estimate for high frequency by taking 
$$T u_0=\Pi_{\infty}\chi e^{-it\paa} u_0.$$ 
To show $T: L^2(X)\rightarrow L^2(\RR_t;\doma^{1/2})$ 
bounded, it is equivalent to show the mapping $TT^*: L^2(\RR_t;\doma^{-1/2}) \rightarrow L^2(\RR_t;\doma^{1/2})$ is bounded, where $T^*$ is given by $$T^*v=\int_{\RR} \Pi_{\infty} e^{is\paa}\chi v ds.$$
Note that 
\begin{equation}
\begin{split}
    TT^*f 
    &=\chi\int_{\RR}\Pi^2_{\infty}e^{-i(t-s)\paa}\chi f ds\\
    &= \chi\int_{s<t}\Pi^2_{\infty}e^{-i(t-s)\paa}\chi f ds + \chi\int_{s>t}\Pi^2_{\infty}e^{-i(t-s)\paa}\chi f ds.
\end{split}
\end{equation}
The boundedness of $TT^*$ is then given by (time-reversed) Duhamel's principle and the estimates \eqref{inhomo_est}. This proves the high-frequency local smoothing estimate.

For the low frequency part, in fact, if $f\in L^2(X)$, then $\Pi_0 f\in \doma^{\infty}$ for a spectral cutoff $\Pi_0$ near $\lambda=0$. Therefore, we have 
\begin{equation}
    e^{-it\paa}\Pi_0 u_0= \Pi_0 e^{-it\paa} u_0 \in L^{\infty}(\RR_t; \doma^{k})
\end{equation}
for all $k\in \mathbb{N}$. This yields
\begin{equation}
    \|\chi e^{-it\paa}\Pi_0 u_0\|_{\doma^{1/2}}=\|(\Id + \paa)^{1/4}\Pi_0 e^{-it\paa} u_0\|_{L^2}\leq C_t\|u_0\|_{L^2}.
\end{equation}
Integrating from $0$ to $T$ yields the local smoothing estimate for low frequency. For $\lambda<-\epsilon$, we use the elliptic estimates. Combining the above three parts together proves the estimate \eqref{local_smoothing}. 
\end{proof}

\begin{remark}
Due to the lack of resolvent estimates for low energy, we are only able to prove this local-in-time local smoothing estimates. In fact, we expect the low energy resolvent estimate to hold and we should be able to obtain a global-in-time local smoothing estimates. (Cf. \cite{yang2021strichartz}.) 
\end{remark}

\appendix
\section{Wave equations and local decay}
\label{app_ffs}
In this appendix, we use Cheeger--Taylor's functional calculus \cite{cheeger1982diffraction} \cite{cheeger1982diffraction2} to construct the forward fundamental solution of the wave equation with one singular vector potential. 

Consider the Aharonov--Bohm Hamiltonian 
\begin{equation}
    P_{\alpha}=(-i\nabla -\Vec{A})^2
\end{equation}
with $\Vec{A}=-\alpha(-\frac{y}{x^2+y^2}, \frac{x}{x^2+y^2})$. In polar coordinates, we have 
\begin{equation}
    \pa=D^2_r-\frac{i}{r}D_r+ \frac{1}{r^2}(D_{\theta}+\alpha)^2.
\end{equation} For well-behaved function $f$, we have 
\begin{equation}
    f(\pa)= (r_1r_2)^{\alpha}\int_0^{\infty}f(\lambda^2)J_{\nu}(\lambda r_1)J_{\nu}(\lambda r_2)\lambda d\lambda
\end{equation}
where $\nu=|D_{\theta}+\alpha|$. By analytic continuation, we can treat $f(\lambda^2)= e^{-\epsilon\lambda}\lambda^{-1} \sin (\lambda t)$ for any $\epsilon>0$. Applying this to the above functional calculus and using the Lipschitz--Hankel integral formula, we obtain the fundamental solution: 
\begin{equation}
    \begin{split}
        \frac{\sin {t\sqrt{\pa}}}{\sqrt{\pa}}
        &= \lim_{\epsilon\rightarrow 0} \im \int_0^{\infty} e^{-(\epsilon+it)\lambda}J_{\nu}(\lambda r_1) J_{\nu}(\lambda r_2) d\lambda\\
        &= -\frac{1}{\pi}(r_1r_2)^{-1/2}\lim_{\epsilon\rightarrow 0} \im Q_{\nu-1/2}\left(\frac{r_1^2+r_2^2+(\epsilon+it)^2}{2r_1r_2}\right)
    \end{split}
\end{equation}
where $Q_{\nu}(z)$ is the Legendre function of the second kind and $\nu=|D_{\theta}+\alpha|$. By the integral formula 
\begin{equation}
    Q_{\nu-1/2}(\cosh{\eta})=\int_{\eta}^{\infty}(2\cosh{s}-2\cosh{\eta})^{-1/2}e^{-s\nu}ds, 
\end{equation}
we obtain that the Schwartz kernel of $\frac{\sin {t\sqrt{\pa}}}{\sqrt{\pa}}$ is equal to
\begin{equation}
\label{ab_fs}
    \begin{cases}
        \displaystyle 0, & \text{ if } t<|r_1-r_2|\\
        \displaystyle \frac{1}{\pi}\int_0^{\beta_1}\left[t^2-(r_1^2+r_2^2-2r_1r_2\cos{s})\right]^{-1/2}\cos(\nu s) ds, & \text{ if } |r_1-r_2|<t< r_1+r_2\\
        \displaystyle \frac{1}{\pi}\cos(\pi\nu)\int_{\beta_2}^{\infty}\left[r_1^2+r_2^2+2r_1r_2\cosh{s}-t^2\right]^{-1/2}e^{-s\nu} ds, & \text{ if } t> r_1+r_2
    \end{cases}
\end{equation}
where 
\begin{equation}
    \beta_1=\cos^{-1}\left(\frac{r_1^2+r_2^2-t^2}{2r_1r_2}\right), \ 
    \beta_2=\cosh^{-1}\left(\frac{t^2-r_1^2-r_2^2}{2r_1r_2}\right). 
\end{equation}
Note that the first equation in \eqref{ab_fs} implies finite speed of propagation and the last equation can be used to deduce a local rate of decay as $t\rightarrow\infty$. 

Now consider the local wave decay. For $t\gg r_1+r_2$, the third integral in the equation \eqref{ab_fs} can be estimated by splitting it into two parts, with the first integrating from $\beta_2$ to $\beta_2+\epsilon$ and the second integrating from $\beta_2+\epsilon$ to $\infty$. Note that 
\begin{equation}
\begin{split}
    \int_{\beta_2}^{\beta_2+\epsilon}\left[r_1^2+r_2^2+2r_1r_2\cosh{s}-t^2\right]^{-1/2}e^{-s\nu} ds \leq C_1\cdot e^{-\beta_2\nu} \leq C_1'\cdot t^{-2\nu}
\end{split}
\end{equation}
where
$$C_1=\int_{\beta_2}^{\beta_2+\epsilon}\left[r_1^2+r_2^2+2r_1r_2\cosh{s}-t^2\right]^{-1/2} ds,$$
and 
\begin{equation}
\begin{split}
    \int_{\beta_2+\epsilon}^{\infty}\left[r_1^2+r_2^2+2r_1r_2\cosh{s}-t^2\right]^{-1/2}e^{-s\nu} ds \leq C_2\cdot e^{-(\beta_2+\epsilon)(\nu-\delta)} \leq C_2'\cdot t^{-2\nu}
\end{split}
\end{equation}
if we choose $\delta$ properly, where
$$C_2=\int_{\beta_2+\epsilon}^{\infty}\left[r_1^2+r_2^2+2r_1r_2\cosh{s}-t^2\right]^{-1/2} e^{-\delta s} ds.$$ 
Combining the above two estimates gives the following local decay of the fundamental solution as $t\rightarrow\infty$. The $t$-derivatives of the fundamental solutions can be estimated similarly. Therefore we have the following proposition:   

\begin{proposition}
For any $\chi(z), \chi'(z')\in \CI_c(\RR^2)$, the Schwartz kernel of $\usin_{\alpha}(t)$ satisfies the local time decay 
\begin{equation}
    |\partial^j_t(\chi\usin_{\alpha}\chi')(t,z,z')|\leq C_{\alpha} t^{-2\nu_0-j+0}, \ j\in\mathbb{N}
\end{equation}
as $t\rightarrow\infty$, where $\nu_0\df \min_{n\in\mathbb{Z}}\{ |n+\alpha| \}$.
\end{proposition}
\begin{remark}
In the above proposition, $-2\nu_0-j+0$ stands for $-2\nu_0-j+\epsilon$ for any $\epsilon>0$.
\end{remark}

\bibliography{reference}
\bibliographystyle{alpha}
\end{document}